\documentclass[a4paper]{amsart}

\usepackage{amsmath,amssymb,amsthm,amscd}
\usepackage[matrix,arrow,curve,tips]{xy}
            \SelectTips{cm}{}

\usepackage{mathrsfs}

\newtheorem{dfn}{Definition}[section]

\newtheorem{prop}[dfn]{Proposition}
\newtheorem{theo}[dfn]{Theorem}
\newtheorem{ex}[dfn]{Example}

\newcommand{\RR}{\mathbb{R}}

\newcommand{\cA}{\mathcal{A}}
\newcommand{\cH}{\mathcal{H}}
\newcommand{\cU}{\mathcal{U}}

\newcommand{\ra}{\rightarrow}

\newcommand{\fb}{\mathfrak{b}}
\newcommand{\fM}{\mathfrak{M}}
\newcommand{\fg}{\mathfrak{g}}

\newcommand{\com}{\mathbin{{\scriptstyle \circ }}}
\newcommand{\ten}{\mathbin{\otimes}}

\newcommand{\id}{\mathord{\mathrm{id}}}

\newcommand{\inv}{\mathord{\mathrm{inv}}}

\newcommand{\cu}{\mathord{\epsilon}}
\newcommand{\cm}{\mathord{\Delta}}
\newcommand{\C}{\mathord{\mathcal{C}^{\infty}}}
\newcommand{\Cc}{\mathord{\mathcal{C}^{\infty}_{c}}}
\newcommand{\G}{\mathord{\Gamma^{\infty}}}
\newcommand{\Gc}{\mathord{\Gamma^{\infty}_{c}}}
\newcommand{\D}{\mathord{{\mathscr{D}}}}
\newcommand{\Prim}{\mathord{{\mathscr{P}}}}
\newcommand{\Univ}{\mathord{{\mathscr{U}}}}

\newcommand{\Simet}{\mathord{{\mathscr{S}}}}
\newcommand{\Bsp}{\mathord{\mathcal{B}_{\mathit{sp}}}}
\newcommand{\Gsp}{\mathord{\mathcal{G}_{\mathit{sp}}}}

\title[]
      {A Cartier-Gabriel-Kostant structure theorem for Hopf algebroids}

\author{J. Kali\v{s}nik}
\address{Institute of Mathematics, Physics and Mechanics,
         University of Ljubljana, Jadranska 19,
         1000 Ljubljana, Slovenia}
\email{jure.kalisnik@fmf.uni-lj.si}

\author{J. Mr\v{c}un}
\address{Department of Mathematics, University of Ljubljana,
         Jadranska 19, 1000 Ljubljana, Slovenia}
\email{janez.mrcun@fmf.uni-lj.si}

\thanks{This work was supported in part by
        the Slovenian Research Agency (ARRS) project J1-2247.}


\begin{document}

\begin{abstract}
In this paper we give an extension of the Cartier-Gabriel-Kostant
structure theorem to Hopf algebroids.
\end{abstract}

\maketitle

\section{Introduction}

For any Hopf algebra $A$, one can consider the associated twisted
tensor product Hopf algebra $\Gamma\ltimes U(\fg)$, where $\Gamma$
is the group of grouplike elements of $A$ and $U(\fg)$ denotes the
universal enveloping algebra of the Lie algebra $\fg$ of primitive
elements of $A$. 
The classical Cartier-Gabriel-Kostant structure theorem
characterizes the Hopf algebras which are isomorphic to their
associated twisted tensor product Hopf algebras.
In this paper we give an extension of this
theorem to Hopf algebroids.

A Hopf $R$-algebroid is a $R/k$-bialgebra
\cite{Nic85,NicWei82,Tak77}, equipped with a bijective antipode.
More recently, structures of this type have been studied in
\cite{BloTaWe08,BloWe08,BoSz04,Kap07,Lu96,Mal00,Xu98}. Most
notably, Hopf algebroids naturally appear as the convolution
algebras of \'{e}tale Lie groupoids \cite{Mrc01,Mrc07_1}: for an
\'{e}tale Lie groupoid $G$ over a manifold $M$, the convolution
algebra $\Cc(G)$ of smooth real functions with compact
support on $G$ is a $\Cc(M)/\RR$-bialgebra with the antipode
induced by the inverse map in the groupoid.
In fact, for any Hopf
$\Cc(M)$-algebroid $A$, the antipode-invariant weakly grouplike
elements of $A$ can be used to construct the associated spectral
\'{e}tale Lie groupoid $\Gsp(A)$ over $M$ \cite{Mrc07_1}.
Furthermore, the
universal enveloping algebra $\Univ(\C(M),L)$ of
a $(\RR,\C(M))$-Lie algebra $L$
is a $\C(M)/\RR$-bialgebra.
If $\fb$ is the Lie algebroid, associated to a bundle of Lie groups over $M$,
then the universal enveloping algebra
$\Univ(\fb)=\Univ(\C(M),\G(\fb))$ of $\fb$ \cite{MoMrc10} 
is a Hopf $\C(M)$-algebroid and in fact a Hopf
algebra over $\C(M)$. For any Hopf $\C(M)$-algebroid $A$, the
space of primitive elements $\Prim(A)$ of $A$ has a structure of a
$(\RR,\C(M))$-Lie algebra.

In this paper we show that if an \'{e}tale Lie groupoid $G$ over
$M$ acts on a bundle of Lie groups $B$ over $M$, then $G$ acts on
the associated universal enveloping algebra $\Univ(\fb)$ of the
Lie algebroid $\fb$ and the associated twisted tensor product
$G\ltimes\Univ(\fb)$ is a Hopf $\Cc(M)$-algebroid. Furthermore, 
in Theorem \ref{theo_Main theorem} we
characterize the Hopf $\Cc(M)$-algebroids of this form:
we give a local condition 
on a Hopf $\Cc(M)$-algebroid $A$ under which
there is a natural isomorphism
\[
A\cong \Gsp(A)\ltimes\Univ(\fb(\Prim(A)))
\]
of Hopf $\Cc(M)$-algebroids, where
$\fb(\Prim(A))$ is a $\Gsp(A)$-bundle of Lie algebras over $M$
with $\Prim(A)=\Gc(\fb(\Prim(A)))$.

\section{Preliminaries}

For the convenience of the reader and to fix the notations, we
will first recall some basic definitions concerning Lie groupoids
\cite{Mac05,MoMrc03,MoMrc05} and Hopf algebroids
\cite{Bo09,Mrc01,Mrc07_1}.

A {\em Lie groupoid} over a smooth (Hausdorff) manifold $M$ is a
groupoid $G$ with objects $M$, equipped with a structure of a
smooth, not-necessarily Hausdorff, manifold such that all the
structure maps of $G$ are smooth, while the source and the target
maps $s,t\!:G\to M$ are submersions with Hausdorff fibers. We
write $G(x,y)=s^{-1}(x)\cap t^{-1}(y)$ for the manifold of arrows
from $x\in M$ to $y\in M$, and $G_{x}=G(x,x)$ for the isotropy
Lie group at $x$.
A Lie groupoid is {\em \'{e}tale} if all its structure maps are local
diffeomorphisms. A {\em bisection} of an \'{e}tale Lie groupoid $G$ is
an open subset $V$ of $G$ such that both $s|_{V}$ and $t|_{V}$ are
injective. Any such bisection $V$ gives a diffeomorphism
$\tau_{V}\!:s(V)\ra t(V)$ by $\tau_{V}=t|_{V}\com (s|_{V})^{-1}$.
The product of bisections $V$ and $W$ of $G$ is the bisection
$V\cdot W=\{gh\,|\,g\in V,\,h\in W,\, s(g)=t(h)\}$, while the inverse
of the bisection $V$ is the bisection $V^{-1}=\{g^{-1}\,|\,g\in V\}$.

A {\em bundle of Lie groups} $B$ over $M$ is a Lie groupoid over $M$
for which the maps $s$ and $t$ coincide.
An {\em action} of a Lie groupoid $G$ over $M$
on a bundle of Lie groups $B$ over $M$ is a smooth action
of $G$ on $B$ along the map $B\ra M$
such that for any $g\in G(x,y)$ the map
$B_{x}\to B_{y}$, $b\mapsto g\cdot b$,
is an isomorphism of Lie groups. A {\em $G$-bundle} of Lie
groups over $M$ is a bundle of Lie groups over $M$ equipped with
an action of $G$. For such a $G$-bundle of Lie groups over $M$
one defines the associated semidirect product Lie
groupoid $G\ltimes B$ over $M$ \cite{MoMrc03} by virtually the
same formulas as those for the semidirect product of groups.
Analogously one
defines $G$-bundles of finite dimensional algebras
(Lie algebras, Hopf algebras) over $M$. We emphasize that these are
required to be locally trivial only as vector bundles. 

We next briefly review the definition of a Hopf algebroid. We
refer the reader to \cite{KalMrc08,Mrc01,Mrc07_1} for more details.
Similar, but inequivalent notions have been studied in
\cite{BloTaWe08,BloWe08,BoSz04,Kap07,Lu96,Mal00,Tak77,Xu98}.

Let $R$ be a commutative (associative, not necessarily unital)
algebra over a field $k$.
A {\em Hopf $R$-algebroid} is a $k$-algebra
$A$ such that $R$ is a commutative, not necessarily central,
subalgebra of $A$ in which $A$ has local units, equipped with a
structure of a left $R$-coalgebra on $A$ (with comultiplication $\cm$
and counit $\cu$) and a $k$-linear involution $S\!:A\ra A$
(antipode) such that
\begin{enumerate}
\item [(i)]   $\cu|_{R}=\id$ and $\cm|_{R}$ is the
              canonical embedding $R\subset
              A\ten_{R}^{ll}A$,

\item [(ii)]  $\cm(A)\subset A\overline{\ten}_{R}A$, where
$A\overline{\ten}_{R}A$ is the algebra consisting of those
elements of $A\ten_{R}A$ on
which both right $R$-actions coincide,

\item [(iii)] $\cu(ab)=\cu(a\cu(b))$ and $\cm(ab)=\cm(a)\cm(b)$
for any $a,b\in A$,

\item [(iv)]   $S|_{R}=\id$ and $S(ab)=S(b)S(a)$ for any
$a,b\in A$,

\item [(v)]  $\mu_{A}\com(S\ten\id)\com\cm=\cu\com S$, where
$\mu_{A}\!:A\ten^{rl}_{R}A\ra A$ denotes the multiplication.
\end{enumerate}
A homomorphism of Hopf $R$-algebroids is defined in the obvious way.

The {\em anchor} of a Hopf $R$-algebroid is the homomorphism of algebras
$\rho\!:A\ra\text{End}_{k}(R)$, given by
$\rho(a)(r)=\cu(ar)$ for $a\in A$ and $r\in R$.
In general, an element $a\in A$ is {\em primitive} if
$\cm(a)=\eta\ten a+a\ten\eta$
for some $\eta\in A$ such that $\eta a=a\eta=a$. If
$A$ is unital, this definition is equivalent with the usual one.
We denote by $\Prim(A)$ the left $R$-module of primitive elements of $A$.
It follows
immediately that $\cu(\Prim(A))=0$. Equipped with the
restriction of the anchor and the natural Lie bracket, the left $R$-module
$\Prim(A)$
becomes a $(k,R)$-Lie algebra \cite{MoMrc10,Rin63}.
Its universal enveloping algebra is denoted by
$\Univ(R,\Prim(A))$.

In this paper, we will mostly focus to the case where
$R$ is the algebra $\Cc(M)$ of smooth real
functions with compact support on a smooth manifold $M$.
We will write $\C(M)_{x}$ for the algebra of germs at a point $x$
of smooth functions on $M$.
Recall that the maximal
ideals of $\Cc(M)$ correspond bijectively to the points of $M$: to
any $x\in M$ we assign the ideal $I_{x}$ of functions which vanish
at $x$. 
The localization of a locally unital
$\Cc(M)$-module (algebra, coalgebra)
$\fM$ at $I_{x}$ is a $\C(M)_{x}$-module (algebra, coalgebra),
which will be denoted by $\fM_{x}$. Note that
$\fM_{x}\cong\fM/N_{x}\fM$, where $N_{x}$ is the ideal of functions
with trivial germ at $x$. An element $m\in\fM$ equals zero if and only if
its localization $m_{x}$ equals zero for any $x\in M$ \cite{Mrc07_2}.
A homomorphism of $\Cc(M)$-modules is bijective
(injective, surjective) if and only if all its localizations are bijective
(injective, surjective). A $\Cc(M)$-module $\fM$ is {\em locally free},
by definition, if $\fM_{x}$ is a free $\C(M)_{x}$-module for
every $x\in M$. A locally free $\Cc(M)$-module $\fM$ is of {\em constant finite rank},
if there exists $n\in\mathbb{N}$ such that $\text{rank}(\fM_{x})=n$ for all $x\in M$.

Let $A$ be a Hopf $\Cc(M)$-algebroid.
An element $a\in A$ is {\em $S$-invariant weakly grouplike} if there
exists $a'\in A$ such that $\cm(a)=a\ten a'$ and
$\cm(S(a))=S(a')\ten S(a)$ (this implies $\cm(S(a))=S(a)\ten
S(a')$). We denote by $G^{S}_{w}(A)$ the set of $S$-invariant
weakly grouplike elements of $A$. An element
$a\in G^{S}_{w}(A)$ is {\em normalized at} $y\in M$ if
$\cu(a)_{y}=1$, and {\em normalized on} $U\subset M$ if it is normalized
at each $y\in U$. Element of the type $a_{y}\in A_{y}$, where
$a\in G^{S}_{w}(A)$ is normalized at $y$, is called an {\em arrow} of
$A$ at $y$. The arrows of $A$ at $y$ form a subset
$G^{S}(A_{y})$ of the set $G(A_{y})$ of grouplike elements of the
$\C(M)_{y}$-coalgebra $A_{y}$. The union of all arrows of $A$
has a natural structure of an \'{e}tale Lie groupoid over $M$
\cite{Mrc07_1}. This groupoid is referred to as the  {\em spectral 
\'{e}tale Lie groupoid} associated to $A$, and denoted by
\[
\Gsp(A)\;.
\]
Each $a\in G^{S}_{w}(A)$, normalized on an open subset $W$ of $M$,
determines a bisection $a_{W}=\{a_{y}\,|\,y\in W\}$ of $\Gsp(A)$,
an open subset $V_{W,a}$ of $M$
and a diffeomorphism $\tau_{W,a}\!:V_{W,a}\ra W$,
implicitly determined by
the equality $fa=a(f\com\tau_{W,a})$ for all $f\in\Cc(W)$.

\section{Convolution Hopf algebroids}

Let $G$ be an \'{e}tale Lie groupoid over $M$ and let
$\cA$ be a $G$-bundle of finite-dimensional unital 
algebras over $\RR$.
We equip the space $\Gc(t^{*}\cA)$ with an
associative convolution product, given by the formula
\[
(ab)(g)=\sum_{g=hk}a(h)(h\cdot b(k))
\]
for any $a,b\in\Gc(t^{*}\cA)$. While only valid for Hausdorff
groupoids, the above formula naturally extends to the
non-Hausdorff case if we use the definition of the space of
sections with compact support
of a vector bundle from \cite{CraMo00}. The corresponding
algebra, called the
{\em convolution algebra of $G$ with coefficients in} $\cA$,
will be denoted by
\[
\Cc(G;\cA)\;.
\]
In particular, if $A$ is a finite dimensional
unital algebra, then we have the convolution
algebra $\Cc(G;A)$
of $G$ with coefficients in the trivial
$G$-bundle with fiber $A$.
Note that $\Cc(G)=\Cc(G;\RR)$.
Since $\cA$ contains the trivial bundle with fiber $\RR$,
the convolution algebra $\Cc(G;\cA)$ contains
the algebra $\Cc(G)$ and henceforth also $\Cc(M)$.
Furthermore, the algebra $\Cc(G;\cA)$ contains $\Cc(M;\cA)$
as well. In fact, the algebra $\Cc(G;\cA)$ 
has local units in $\Cc(M)$ and is
generated by the sum of $\Cc(G)$ and $\Cc(M;\cA)$.
Note also that $\Cc(M;\cA)$ is a subalgebra
of the algebra $\G(\cA)=\C(M;\cA)$ of sections
of $\cA$ with the pointwise multiplication.

It is often inconvenient to compute products of arbitrary
functions. Things get simplified if we restrict the calculations
to the functions with supports in bisections. In particular, this is
the easiest way of defining the convolution product on
$\Cc(G;\cA)$ if $G$ is non-Hausdorff. The bisections of $G$ form a
basis for the topology on $G$, so we can decompose an arbitrary
function $a\in\Cc(G;\cA)$ as a sum $a=\sum_{i=1}^{n}a_{i}$, where
each $a_{i}$ has its support in a bisection $V_{i}$ of $G$. In the
non-Hausdorff case, this is in fact the definition of
$\Cc(G;\cA)$.

Any function $a\in\Cc(G;\cA)$ with support in a bisection $V$ can
be written in the form $a=a_{0}\com t|_{V}$ for a unique
$a_{0}\in\Cc(t(V);\cA)$. Moreover, if $a'\in\Cc(V)$ is a function
equal to $1$ on a neighbourhood of the support of $a$, we can express
$a$ as the product $a_{0}a'$ 
(we say that such a decomposition $a_{0}a'$ is {\em standard
decomposition} of $a$).
The bisection $V$ induces an isomorphism of bundles of algebras
$(\mu_{V},\tau_{V})\!:\cA|_{s(V)}\ra\cA|_{t(V)}$.
For any $c_{0}\in\C(M;\cA)$ we define the function
$\mu_{V}(c_{0})\in\C(t(V);\cA)$ by
\[ \mu_{V}(c_{0})=\mu_{V}\com c_{0}|_{s(V)}\com\tau_{V}^{-1}\;.
\]
Let $b\in\Cc(G;\cA)$ be another function with support in a bisection
of $G$, say $W$, and with a standard decomposition $b_{0}b'$.
Observe that
\begin{equation}\label{eq_convolution product}
ab=(a_{0}\com t|_{V})(b_{0}\com
t|_{W})=(a_{0}\mu_{V}(b_{0}))\com t|_{V\cdot W}\;,
\end{equation}
where $a_{0}\mu_{V}(b_{0})$ denotes the product of functions
in $\C(t(V);\cA)$. The product in $\Cc(G;\cA)$ can be therefore
expressed as a combination of the multiplication in
$\C(M;\cA)$ and the action of $G$ on $\cA$. Since the algebra
$\Cc(G;\cA)$ contains $\Cc(M)$ as a commutative subalgebra,
we can consider $\Cc(G;\cA)$ as a left $\Cc(M)$-module. The
action of $\Cc(M)$ on $\Cc(G;\cA)$ can be expressed as a
scalar multiplication along the fibers of the map $t$. We will
often use the isomorphism
\begin{equation}\label{eq_isomorphism of modules of sections}
\Cc(G;\cA)\cong\C(M;\cA)\ten_{\Cc(M)}\Cc(G)
\end{equation}
of left $\Cc(M)$-modules \cite{GrHaVa78}. 
Under this isomorphism, the function
$a$ corresponds to the tensor $a_{0}\ten a'$, while
the equation (\ref{eq_convolution product}) translates to
\[
(a_{0}\ten a')(b_{0}\ten b')=a_{0}\mu_{V}(b_{0})\ten a'b'\;.
\]

\subsection{The twisted tensor product Hopf algebroid}

Let $\cH$ be a $G$-bundle of finite dimensional
Hopf algebras (with involutive antipodes).
The space $\C(M;\cH)$ is a Hopf algebra over
$\C(M)$. Write $\cm_{0}$, $\cu_{0}$ and $S_{0}$
for the comultiplication,
the counit and the antipode
of $\C(M;\cH)$ respectively.
Since $\Cc(G)$ is a coalgebra over
$\Cc(M)$ \cite{Mrc07_1}, the tensor product
$\C(M;\cH)\ten_{\Cc(M)}\Cc(G)$,
which we identify with $\Cc(G;\cA)$
by (\ref{eq_isomorphism of modules of sections}),
naturally becomes a
coalgebra over $\Cc(M)$. The coalgebra structure is
given by
\[
\cm(a_{0}\ten a')=\sum_{i=1}^{n}(a_{0}^{i,1}\ten a')\ten(a_{0}^{i,2}\ten a')\;,
\]
\[
\cu(a_{0}\ten a')=\cu_{0}(a_{0})\;,
\]
for any function $a\in\Cc(G;\cA)$ with support in a bisection $V$
and with standard decomposition $a_{0}a'$, where
$\cm_{0}(a_{0})=\sum_{i=1}^{n}a_{0}^{i,1}\ten a_{0}^{i,2}$
and $a_{0}^{i,1}$, $a_{0}^{i,2}$ are chosen so that
$a_{0}^{i,1}a'$ and $a_{0}^{i,2}a'$ are standard decompositions.
Furthermore, there is the antipode on $\C(M;\cH)\ten_{\Cc(M)}\Cc(G)$
given by
\[
S(a_{0}\ten a')=\mu_{V^{-1}}(S_{0}(a_{0}))\ten S_{G}(a')\;,
\]
where $S_{G}$ is the antipode on $\Cc(G)$. 
With respect to the isomorphism (\ref{eq_isomorphism of modules of sections}),
we can express $S$ also in the form
$S(a)(g)=g\cdot S_{\cH}(a(g^{-1}))$,
where $S_{\cH}$ denotes the antipode on $\cH$.

\begin{prop}\label{prop_Hopf algebroid}
The convolution algebra $\Cc(G;\cH)$, together with the structure maps
$(\cm,\cu,S)$ defined above, is a Hopf $\Cc(M)$-algebroid.
\end{prop}
\begin{proof}
The axioms can be verified by direct computations.
Let us show, for example, that
$S(ab)=S(b)S(a)$ holds for any $a,b\in\Cc(G;\cH)$. We can assume
without loss of generality that the function $a$ has support in a bisection $V$,
the function $b$ has support in a bisection $W$ and $s(V)=t(W)$.
Write $a_{0}a'$ and $b_{0}b'$ for standard decompositions of $a$ and $b$ respectively. 
Then
\begin{align*}
S(ab)  & =S\left((a_{0}\ten a')(b_{0}\ten b')\right)=S(a_{0}\mu_{V}(b_{0})\ten a'b')  \\
   & =\mu_{(V\cdot W)^{-1}}\left(S_{0}(a_{0}\cdot\mu_{V}(b_{0}))\right)\ten S_{G}(a'b') \\
   & =\mu_{(V\cdot W)^{-1}}\left(S_{0}(\mu_{V}(b_{0}))S_{0}(a_{0})\right)\ten S_{G}(a'b') \\
   & =\mu_{W^{-1}}\left(\mu_{V^{-1}}(S_{0}(\mu_{V}(b_{0})))\right)\mu_{(V\cdot W)^{-1}}(S_{0}(a_{0}))\ten S_{G}(a'b') \\
   & =\mu_{W^{-1}}(S_{0}(b_{0}))\mu_{W^{-1}}(\mu_{V^{-1}}(S_{0}(a_{0})))\ten S_{G}(b')S_{G}(a') \\
   & =\left(\mu_{W^{-1}}(S_{0}(b_{0}))\ten S_{G}(b')\right)\left(\mu_{V^{-1}}(S_{0}(a_{0}))\ten S_{G}(a')\right) \\
   & =S(b)S(a)\;. \qedhere
\end{align*}
\end{proof}

Note that both $\Cc(G)$ and
$\Cc(M;\cH)=\C(M;\cH)\ten_{\Cc(M)}\Cc(M)$
are Hopf $\Cc(M)$-subalgebroids of $\Cc(G;\cH)$.
The Hopf algebroid $\Cc(G;\cH)$ will be
referred to as the {\em twisted tensor product} of $G$ and $\C(M;\cH)$,
and denoted by
\[
G\ltimes\C(M;\cH)\;.
\]
This is in particular motivated by the following example:

\begin{ex}\rm
The convolution Hopf algebra of an action of a discrete group
$\Gamma$ on a Hopf algebra $H$ is isomorphic to the twisted tensor
product Hopf algebra $\Gamma\ltimes H$.
\end{ex}

Our definition of a Hopf algebroid roughly corresponds to the
notion of a left Hopf algebroid with antipode in
\cite{Bo09,BoSz04,KowPos09}. However, in the unital case where $M$ is compact,
the structure of $G\ltimes\C(M;\cH)$ satisfies axioms of a Hopf
algebroid given in \cite{Bo09,BoSz04,KowPos09}.

\subsection{The Hopf algebroid associated to a semidirect product}
\label{The Hopf algebroid associated to a semidirect product}

Let $B$ be a bundle of connected Lie groups over $M$,
equipped with a left action of an \'{e}tale Lie groupoid $G$ over $M$.
In this Subsection we will define the Hopf $\Cc(M)$-algebroid
associated to the semidirect product Lie groupoid $G\ltimes B$.
The definition is a slight extension of the one in previous Subsection,
since we have to consider the convolution algebra with
coefficients in a bundle of infinite dimensional Hopf algebras,
filtered with subbundles of finite rank.
In the extreme trivial cases, this definition extends the definition
of $\Cc(M)$-Hopf algebroid $\Cc(G)$ associated to $G$ \cite{Mrc07_1}
as well as the definition of $\Cc(M)$-Hopf algebroid of
the Lie algebroid associated to $B$ \cite{MoMrc10}.

Denote by $\fb$ the bundle of Lie algebras associated to $B$.
The universal enveloping algebra $\Univ(\fb)$ of the Lie algebroid $\fb$
over $M$ is not only a $\C(M)/\RR$-bialgebra \cite{MoMrc10},
but also a Hopf $\C(M)$-algebroid and in fact a Hopf
algebra over $\C(M)$. 
To each Lie algebra $\fb_{x}$, the fiber of $\fb$ over
a point $x\in M$, we
naturally assign its universal enveloping algebra $U(\fb_{x})$,
together with its natural filtration
$U(\fb_{x})^{(0)}\subset\cdots\subset U(\fb_{x})^{(k)}\subset
U(\fb_{x})^{(k+1)}\subset\cdots$. The family of vector spaces
\[
\cU(\fb)^{(k)}=\coprod_{x\in M}U(\fb_{x})^{(k)}
\]
can be, for each $k$, equipped with a smooth vector bundle
structure over $M$ (by considering local trivializations obtained
from PBW-theorem). Define a family of (infinite dimensional)
vector spaces over $M$
\[
\cU(\fb)=\lim_{\ra}\cU(\fb)^{(k)}\;,
\]
with fiber over $x$ being the Hopf algebra $U(\fb_{x})$. The space
of smooth sections of $\cU(\fb)$ is defined
as
\[
\C(M;\cU(\fb))=\G(\cU(\fb))=\lim_{\ra}\G(\cU(\fb)^{(k)})\;.
\]
The structure maps on the fibers of $\cU(\fb)$ extend to the structure
maps on the space $\C(M;\cU(\fb))$ and turn $\C(M;\cU(\fb))$ into a Hopf
algebra over $\C(M)$.
Note that there is a natural isomorphism 
$\Univ(\fb)\cong\C(M;\cU(\fb))$
of Hopf algebras over $\C(M)$.

Since the groupoid $G$ acts on $B$, the bundle $\fb$ is
in fact a $G$-bundle of Lie algebras. The action of $G$ on $\fb$
extends to an action on the family of Hopf
algebras $\cU(\fb)$: any arrow $g\in G(x,y)$ induces
an isomorphism of Hopf algebras $U(\fb_{x})\ra U(\fb_{y})$, which,
in particular, preserves the natural filtration.
We define
\[
\Cc(G;\cU(\fb))=\lim_{\longrightarrow}\Gc(t^{*}(\cU(\fb)^{(k)}))\;.
\]
Again, we have a natural isomorphism of left $\Cc(M)$-modules
\[ 
\Cc(G;\cU(\fb))\cong \C(M;\cU(\fb))\ten_{\Cc(M)}\Cc(G) \;.
\]
We equip $\Cc(G;\cU(\fb))$
with multiplication, unit, comultiplication,
counit and antipode as in Proposition \ref{prop_Hopf algebroid},
to obtain the {\em twisted tensor product} Hopf $\Cc(M)$-algebroid
\[
G\ltimes\Univ(\fb)=\Cc(G;\cU(\fb))
\]
associated to the semidirect product groupoid $G\ltimes B$. The
construction of the Hopf $\Cc(M)$-algebroid $G\ltimes\Univ(\fb)$ is
functorial with respect to the smooth functors over $\id_{M}$.

\begin{ex}\rm
(1) Let $G\ltimes B$ be a semidirect product of a discrete group
$G$ and a connected Lie group $B$. The associated Hopf algebroid is in
this case isomorphic to the twisted tensor product Hopf algebra
$G\ltimes U(\fb)$. We consider the Hopf algebra $G\ltimes U(\fb)$
associated to the semidirect product $G\ltimes B$ as a
geometrically constructed Hopf algebra which generalizes both
the group algebras and the universal enveloping algebras.

(2) If $B$ is trivial, then $G\ltimes\Univ(\fb)$ is equal to
$\Cc(G)$.

(3) A $G$-bundle of vector spaces is naturally a $G$-bundle of Lie
algebras with trivial bracket. For such a $G$-bundle of vector spaces
$B$ we obtain the convolution Hopf algebroid $G\ltimes \Simet(B)$,
where $\Simet(B)$ is the symmetric algebra of the module $\G(B)$.
\end{ex}

The Hopf algebroid $G\ltimes\Univ(\fb)$ of a semidirect product
$G\ltimes B$ can be naturally represented by a certain class of
partial differential operators on the groupoid $G\ltimes B$. Note
first that the Lie algebroid associated to $G\ltimes B$ is in fact
equal to the Lie algebroid $\fb$ of $B$ since $G$ is \'{e}tale. We
can therefore consider elements of $\Gc(\fb)$ as right invariant
vector fields on $G\ltimes B$ \cite{MoMrc03,NiWeXu99}. In the same
manner, we assign to any $a\in  \Gc(t^{*}(\fb)) \subset
G\ltimes\Univ(\fb)$ the vector field $X_{a}$ on $G\ltimes B$,
given by
\[
X_{a}(b,g)=dR_{(b,g)}(a(g))\;.
\]
Such a vector field is $B$-invariant by construction and
completely determined by its values on the subgroupoid $G$ of
$G\ltimes B$. The support of the vector field $X_{a}$ is in
general not compact, if the fibers of $B$ are not compact.
However, it makes sense to define the support of a $B$-invariant
vector field on $G\ltimes B$ as a subset of $G$ and not of the
whole $G\ltimes B$. In this way, the vector field $X_{a}$ has
a compact support. By generalizing the above construction, an
arbitrary element of $G\ltimes\Univ(\fb)$ thus corresponds to a
$B$-invariant partial differential operator on $G\ltimes B$ with
compact support.

\begin{ex}\rm
(1) Let $G\ltimes B$ be a semidirect product of a discrete group
$G$ and a connected Lie group $B$. An arbitrary element $D\in
G\ltimes U(\fb)$ can be written as a finite sum $D=\sum_{g\in
G}D_{g}\delta_{g}$, where $D_{g}\in U(\fb)\cong
\text{PDO}_{\inv}(G\ltimes B)$ \cite{NiWeXu99} and
$\delta_{g}$ is a function on $G$ which is equal to $1$ at $g$ and
is $0$ everywhere else. Viewed as a partial differential operator
on $G\ltimes B$, $D$ equals to $D_{g}$ on the connected component
of $G\ltimes B$ corresponding to $g$. It is nonzero only on
finitely many connected components of $G\ltimes B$. Denote now by
$\mathcal{D}'(G\ltimes B)$ the convolution algebra of
distributions with compact support on $G\ltimes B$ \cite{CdSW99}.
We can faithfully represent the algebra $G\ltimes U(\fb)$ into
$\mathcal{D}'(G\ltimes B)$ by assigning to $D_{g}\delta_{g}\in
G\ltimes U(\fb)$ the distribution which corresponds to the
distributional derivative of
$\delta_{(1,g)}\in\mathcal{D}'(G\ltimes B)$ along $D_{g}$.

(2) For a general semidirect product groupoid $G\ltimes B$ we
consider
\[
G\ltimes\Univ(\fb)\cong\text{PDO}_{B\text{-inv},c}(G\ltimes B)
\]
as the space of $B$-invariant partial differential operators on
$G\ltimes B$ with compact support and with convolution product.
\end{ex}

\section{The structure of Hopf algebroids}

The aim of this section is to characterize the convolution Hopf
algebroids of semidirect products of \'{e}tale Lie groupoids and
of bundles of Lie groups.

\subsection{The algebra $\D(A)$}

We start by exploring some properties of the space $\Prim(A)$ of
primitive elements of a Hopf $R$-algebroid
$A$ and its relation to the base algebra
$R$ and the antipode $S$. The well known identity $S(X)=-X$
for $X\in\Prim(A)$ does not hold for general Hopf algebroids.
Concrete counter-examples, with geometric origin, have been
constructed and described in \cite{KowPos09}. In our case,
however, we will be mostly interested in those Hopf algebroids
for which the space of primitive elements is $S$-invariant. Some
of the properties of such algebroids are described in the
following propositions.

\begin{prop}\label{prop_S-invariant subspace of primitive elements}
Let $A$ be a Hopf $R$-algebroid. The following statements are
equivalent:
\begin{enumerate}
\item [(i)] $S(\Prim(A))=\Prim(A)$.

\item [(ii)] $S(\Prim(A))\subseteq\Prim(A)$.

\item [(iii)] For every $X\in\Prim(A)$ we have $S(X)=-X$.
\end{enumerate}
\end{prop}
\begin{proof}
The implications (i)$\Rightarrow$(ii) and (iii)$\Rightarrow$(i)
are immediate. Suppose now that $S(\Prim(A))\subseteq\Prim(A)$. To
show (iii), we use the equality
$\mu_{A}\com(S\ten\id)\com\cm=\cu\com S$ on an element
$X\in\Prim(A)$ and obtain
\[
S(X)+X=\cu(S(X))\;.
\]
Since $\Prim(A)$ is $S$-invariant, we have $S(X)\in\Prim(A)$,
which implies $\cu(S(X))=0$.
\end{proof}

\begin{prop}\label{prop_Hopf algebroids with trivial anchor}
Let $A$ be a Hopf $R$-algebroid. The following statements are
equivalent:
\begin{enumerate}
\item [(i)] Elements of $\Prim(A)$ commute with elements of $R$.

\item [(ii)] The space $\Prim(A)$ is a right $R$-submodule of $A$.

\item [(iii)] The $(k,R)$-Lie algebra $\Prim(A)$ has trivial anchor.
\end{enumerate}
\end{prop}
\begin{proof}
(i)$\Rightarrow$(ii) $\Prim(A)$ is always a left
$R$-submodule of $A$, hence in this case it is also a right
$R$-submodule.

(ii)$\Rightarrow$(iii) Take any $r\in R$ and any
$X\in\Prim(A)$. From $Xr\in\Prim(A)$ it follows $\cu(Xr)=0$ and
therefore $\rho(X)(r)=\cu(Xr)=0$.

(iii)$\Rightarrow$(i) Take any $r\in R$, any
$X\in\Prim(A)$ and let $\eta\in R$ be a local unit for both $r$ and $X$. Then
\[
\cm(Xr)=\cm(X)\cm(r)=(X\ten\eta+\eta\ten X)(\eta\ten r)=X\ten r+\eta\ten Xr\;.
\]
By applying the map $\id\ten\cu$ on both sides we obtain
\[
Xr=rX+\cu(Xr)\;.
\]
If the anchor is trivial, we therefore have $rX=Xr$.
\end{proof}

The space of primitive elements $\Prim(A)$ is in general not a Lie
algebra over $R$. This is true, however, if the anchor of $\Prim(A)$
is trivial, as follows from Proposition
\ref{prop_Hopf algebroids with trivial anchor}.

\begin{prop}
Let $A$ be a Hopf $R$-algebroid. If $\Prim(A)$ is
$S$-invariant, then the $(k,R)$-Lie algebra  $\Prim(A)$ has
trivial anchor.
\end{prop}
\begin{proof}
Choose any $r\in R$ and any $X\in\Prim(A)$. We have to show
$Xr\in\Prim(A)$. First note that
\[
Xr=S(S(Xr))=S(rS(X))\;.
\]
Since $\Prim(A)$ is $S$-invariant and a left $R$-module, we
get $S(rS(X))\in\Prim(A)$.
\end{proof}

We will assume from now on that $M$ is a Hausdorff manifold and that $A$ is a 
Hopf $\Cc(M)$-algebroid. Part of the structure of a Hopf algebroid $A$ is its anchor, 
which defines an action of $A$ on the algebra $\Cc(M)$. In this respect, two classes
of elements of $A$ will be of particular interest for us.
Primitive elements of $A$ act on $\Cc(M)$ by derivations (which
correspond to vector fields on $M$). It is therefore convenient to
consider the subalgebra $\D(A)$ of $A$, generated by $\Cc(M)$ and
$\Prim(A)$, which acts on $\Cc(M)$ by partial differential
operators. Under some mild assumptions we can identify $\D(A)$
with the universal enveloping algebra $\Univ(\Cc(M),\Prim(A))$ of
the $(\RR,\Cc(M))$-Lie algebra $\Prim(A)$
(see Proposition \ref{prop_injective map from U(P(A))}).

Another important subset $G^{S}_{w}(A)$ of $A$ consists of
$S$-invariant weakly grouplike elements. The action of such an element
$a$ on $\Cc(M)$ was studied in \cite{Mrc07_1},
where it was used to define the spectral \'{e}tale Lie groupoid associated to $A$.
Note that the anchor $\rho(a)$ corresponds to the operator
$T_{S(a)}$ (see \cite{Mrc07_1}) given by
\[
\rho(a)(f)=T_{S(a)}(f)=\cu(af)
\]
for any $f\in\Cc(M)$. Alternatively, one can
equivalently define $\rho(a)(f)=afS(a')$, where $a'$ is any
element of $A$
such that $\cm(a)=a\ten a'$ and $\cm(S(a))=S(a')\ten S(a)$.
By using this definition, the operator $\rho(a)$ can
be extended to the whole $A$.

We say that a pair of elements $a,a'\in G^{S}_{w}(A)$ is
a {\em good pair}, if there exist an element $c\in G^{S}_{w}(A)$,
normalized on an open subset $U$ of $M$, and functions
$f,f'\in\Cc(U)$ such that $a=fc$, $a'=f'c$ and
$f'$ equals $1$ on an open neighbourhood of the support of $f$.
Such an element $c$ will be called a {\em witness} for
the good pair $a,a'$.
Observe that $f=\cu(a)$, $f'=\cu(a')$, $\cm(a)=a\ten a'=a\ten c$,
$\cm(S(a))=S(a')\ten S(a)=S(c)\ten S(a)$,
$\cm(a')=a'\ten c$,
$\cm(S(a'))=S(c)\ten S(a')$,
$f=aS(a')=aS(c)$ and $f'=a'S(c)$ (see \cite{Mrc07_1}).
We say that an element $d\in G^{S}_{w}(A)$ is {\em good},
if there exists  $d'\in G^{S}_{w}(A)$ such that
$d,d'$ is a good pair.

For every good pair $a,a'\in G^{S}_{w}(A)$ we define an $\RR$-linear
operator $T_{a,a'}\!:A\ra A$ by
\[
T_{a,a'}(b)=abS(a')
\]
for any $b\in A$. Its restriction to $\Cc(M)$ clearly equals
$\rho(a)$.

\begin{prop}\label{prop_invariance of cP(A)}
Let $A$ be a Hopf $\Cc(M)$-algebroid and let $a,a'\in G^{S}_{w}(A)$
be a good pair.
\begin{enumerate}
\item [(i)] The base subalgebra $\Cc(M)$ of $A$ is $T_{a,a'}$-invariant.

\item [(ii)] The $(\RR,\Cc(M))$-Lie algebra $\Prim(A)$ is $T_{a,a'}$-invariant.

\item [(iii)] The subalgebra $\D(A)$ of $A$ is $T_{a,a'}$-invariant.
\end{enumerate}
\end{prop}
\begin{proof}
(i) This follows from the equality
$T_{a,a'}(f)=\cu(af)$, which holds for any $f\in\Cc(M)$.

(ii) Let $c$ be a witness for the good pair $a,a'$.
By definition, there exists an open subset $U$ of $M$ such that
$c$ is normalized on $U$ and $\cu(a),\cu(a')\in\Cc(U)$.
Write $f=\cu(a)$ and $f'=\cu(a')$.
Choose any $X\in\Prim(A)$ and let $\eta\in\Cc(M)$ be a local unit
for $X$, $T_{a,a'}(X)$, $c$, $f$ and $f'$. It follows that
\begin{align*}
\cm(T_{a,a'}(X))& =\cm(aXS(a'))\\
                & =\cm(fc)\cm(X)\cm(S(f'c))  \\
                & =(fc\ten f'c)(X\ten\eta+\eta\ten X)(S(c)\ten S(f'c)) \\
                & = fcXS(c)\ten f'cS(f'c)+fcS(c)\ten f'cXS(f'c) \\
                & = f'cXS(c)\ten fcS(f'c)+f'cS(c)\ten fcXS(f'c) \\
                & = f'cXS(c)\ten f+f'\ten fcXS(f'c) \\
                & = ff'cXS(c)\ten \eta + \eta\ten f'fcXS(f'c) \\
                & = aXS(c)\ten\eta +\eta\ten aXS(a')\;.
\end{align*}
The second term in the last line is equal to $\eta\ten T_{a,a'}(X)$.
We need to see that the first term in the last line equals
$T_{a,a'}(X)\ten \eta$. Indeed, we have
\begin{align*}
aXS(a')  &=fcXS(f'c)  \\
         &=fcXS(c(f'\com\tau_{U,c})) \\
         &=fcX(f'\com\tau_{U,c})S(c)  \\
         &=fc \big( (f'\com\tau_{U,c})X + \rho(X)(f'\com\tau_{U,c}) \big) S(c)  \\
         &=fc (f'\com\tau_{U,c})X S(c) + fc \rho(X)(f'\com\tau_{U,c}) S(c)  \\
         &=ff'cXS(c) + c(f\com\tau_{U,c})\rho(X)(f'\com\tau_{U,c})S(c)  \\
         &=aXS(c)+ c(f\com\tau_{U,c})\rho(X)(f'\com\tau_{U,c})S(c)\;.
\end{align*}
From the fact that $\rho(X)$ is a derivation on $\Cc(M)$ and from the equality
$ff'=f$ it follows that
$(f\com\tau_{U,c})\rho(X)(f'\com\tau_{U,c})=0$,
thus $T_{a,a'}(X)=aXS(a')=aXS(c)$.

(iii) 
Choose a function $f''\in\Cc(U)$ which equals $1$ on
an open neighbourhood $W''$ of the support of $f'$, and put $a''=f''c$.
Note that $a''$ is another witness for the good pair $a,a'$ and
that both $a,a''$ and $a',a''$ are good pairs with witness $c$.
Write $\tau=\tau_{W'',a''}$.

The vector space $\D(A)$ is generated
by elements of the form $X_{1}X_{2}\ldots X_{k}\phi$, where
$X_{i}\in\Prim(A)$ and $\phi\in\Cc(M)$.
It is therefore enough to show
that $T_{a,a'}$ maps all such elements into $\D(A)$.

We will prove this by induction on $k$. For $k=0$ this is true by (i).
Now assume that $T_{d,d'}(X_{1}X_{2}\cdots X_{k-1}\phi)\in \D(A)$ 
for any good pair $d,d'\in G_{w}^{S}(A)$ and
any $\phi\in\Cc(M)$. We only have to show that, under this
induction hypothesis, we have $T_{a,a'}(X_{1}X_{2}\cdots X_{k-1}X_{k}\phi)\in\D(A)$
for any $\phi\in\Cc(M)$.
To this end, note that the equalities
$a'=\cu(a')a''=a''(\cu(a')\com\tau)$
and $\cu(a')\com\tau=S(a'')a'$
imply
\begin{align*}
T_{a,a'}(X_{1}X_{2}\cdots X_{k}\phi) &= a X_{1}X_{2}\cdots X_{k}\phi S(a') \\
     &= aX_{1}X_{2}\cdots X_{k}\phi(\cu(a')\com\tau)S(a'')  \\
     &= aX_{1}X_{2}\cdots X_{k-1}\phi(\cu(a')\com\tau)X_{k}S(a'') \\
     & \quad + aX_{1}X_{2}\cdots X_{k-1} \rho(X_{k})(\phi(\cu(a')\com\tau)) S(a'') \\
     &= aX_{1}X_{2}\cdots X_{k-1}S(a'')a'\phi X_{k}S(a'') \\
     & \quad + aX_{1}X_{2}\cdots X_{k-1} \rho(X_{k})(\phi(\cu(a')\com\tau)) S(a'') \\
     &=T_{a,a''}(X_{1}X_{2}\cdots X_{k-1}) T_{a',a''}(\phi X_{k}) \\
     &\quad +T_{a,a''}(X_{1}X_{2}\cdots X_{k-1} \rho(X_{k})(\phi(\cu(a')\com\tau)))\;.
\end{align*}
The result now follows by (i) and the induction hypothesis.
\end{proof}

In general, $T_{a,a'}$ is not a multiplicative map. For our
purposes, it will suffice that $T_{a,a'}$ is in a certain
sense locally multiplicative.

For any $\Cc(M)$-submodule $B$ of $A$ and any open subset $U$ of $M$, define
\[
B|_{U}=\{ fb\,|\, f\in\Cc(U),\,b\in B \}\;.
\]
Note that an element $b\in B$ belongs to $B|_{U}$
if and only if there exists $f\in\Cc(U)$ such
that $fb=b$. Observe also that
$A|_{U}$ and $\D(A)|_{U}$ are subalgebras of $A$, while
$\Prim(A)|_{U}$ is a $(\RR,\Cc(M))$-Lie subalgebra of $\Prim(A)$.

The algebra $\D(A)|_{U}$ is generated by
$\Cc(U)$ and $\Prim(A)|_{U}$ together. Commutation relations between
$\Cc(U)$ and $\Prim(A)|_{U}$ show that for any $D\in\D(A)|_{U}$
there exists $f\in\Cc(U)$ such that $fD=Df=D$. Moreover, any
$\eta\in\Cc(M)$ which equals $1$ on $U$ acts as a two-sided unit on
$\D(A)|_{U}$. 

Let $a\in G^{S}_{w}(A)$ be good, normalized
on an open subset $W$ of $M$,
and let $\tau_{W,a}\!:V_{W,a}\ra W$ be the
corresponding diffeomorphism.
Choose $a'\in G^{S}_{w}(A)$ so that $a,a'$ is a good pair.
In particular, the element $a'$ is normalized on an open
neighbourhood $W'$ of the support of $\cu(a)$, while
$\cu(a)\com\tau_{W',a'}$ equals to $1$ on $V_{W,a}$. 
For any $D\in\D(A)|_{V_{W,a}}$ we have
\begin{align*}
T_{a,a'}(D) & =aDS(a')=aD(\cu(a)\com\tau_{W',a'})S(a')=aDS(a) \\
            & =\cu(a)a'DS(a)=a'(\cu(a)\com\tau_{W',a'})DS(a)=a'DS(a)\;.
\end{align*}
In particular,
the restriction of $T_{a,a'}$ to $\D(A)|_{V_{W,a}}$ depends only
on $W$ and $a$, and not on the choice of $a'$.
We will therefore denote this restriction by $T_{W,a}$.

\begin{prop}\label{prop_isomorphism T}
Let $A$ be a Hopf $\Cc(M)$-algebroid and let
$a\in G^{S}_{w}(A)$ be good, normalized on an open subset
$W$ of $M$. The map $T_{W,a}$
restricts to
\begin{enumerate}
\item [(i)] the algebra isomorphism
$(\tau_{W,a}^{-1})^{*}\!:\Cc(V_{W,a})\ra\Cc(W)$,

\item [(ii)] an isomorphism of Lie algebras
$\Prim(A)|_{V_{W,a}}\ra\Prim(A)|_{W}$, and

\item [(iii)] an isomorphism of algebras
$\D(A)|_{V_{W,a}}\ra\D(A)|_{W}$.
\end{enumerate}
\end{prop}
\begin{proof}
(i) This part was proven in \cite{Mrc07_1} and is in fact
the definition of $\tau_{W,a}$.

(iii) First we show that $T_{W,a}\left(\D(A)|_{V_{W,a}}\right)\subset\D(A)|_{W}$.
Choose $a'\in G^{S}_{w}(A)$ so that $a,a'$ is a good pair.
Let $D\in\D(A)|_{V_{W,a}}$ and choose
$f\in\Cc(V_{W,a})$ such that $fD=D$. Then we have
\[
T_{W,a}(D)=aDS(a')=afDS(a')=(f\com\tau_{W,a}^{-1})aDS(a')\in\D(A)|_{W}\;.
\]

Next, we show that the restriction of $T_{W,a}$ to
$\D(A)|_{V_{W,a}}$ is multiplicative. The function $S(a')a\in\Cc(M)$
equals to $1$ on $V_{W,a}$, thus it acts as a two-sided unit on
$\D(A)|_{V_{W,a}}$. For any
$D_{1},D_{2}\in\D(A)|_{V_{W,a}}$ it follows
\[
T_{W,a}(D_{1}D_{2})=aD_{1}D_{2}S(a')=aD_{1}S(a')aD_{2}S(a')
=T_{W,a}(D_{1})T_{W,a}(D_{2})\;.
\]
By replacing $a$ with $S(a)$ in the above arguments we see that
$T_{V_{W,a},S(a)}$ restricts to a homomorphism of algebras
$\D(A)|_{W}\ra\D(A)|_{V_{W,a}}$. Since $aDS(a')=a'DS(a)$ for any
$D\in\D(A)|_{V_{W,a}}$, we have
\[
T_{V_{W,a},S(a)}(T_{W,a}(D))=S(a)aDS(a')a'=S(a)a'DS(a)a'\;.
\]
The function $S(a)a'\in\Cc(M)$ equals to $1$ on $V_{W,a}$, hence
$T_{V_{W,a},S(a)}(T_{W,a}(D))=D$. Analogous arguments show
that $T_{W,a}\com T_{V_{W,a},S(a)}=\id_{\D(A)|_{W}}$.

(ii) This follows from (iii).
\end{proof}

\subsection{Spectral semidirect product Lie groupoid}
\label{Spectral semidirect product Lie groupoid}

In Subsection \ref{The Hopf algebroid associated to a semidirect product}
we have constructed the Hopf algebroid $G\ltimes\Univ(\fb)$ associated to
a semidirect product $G\ltimes B$ of an \'{e}tale Lie groupoid $G$ and a bundle
of connected Lie groups $B$. Our aim now is to describe to what extent
the Lie groupoid $G\ltimes B$ can be reconstructed
from $G\ltimes\Univ(\fb)$. We will show how to assign to a Hopf $\Cc(M)$-algebroid,
satisfying certain conditions, its spectral semidirect product Lie groupoid.

We will assume from now on that $A$ is a Hopf $\Cc(M)$-algebroid 
which is locally free as a left $\Cc(M)$-module. 
If the $(\RR,\Cc(M))$-Lie algebra $\Prim(A)$ is $S$-invariant
and locally free of constant finite rank as a $\Cc(M)$-module 
(if $M$ is compact, this last condition
is equivalent to $\Prim(A)$ being finitely generated
and projective as a $\Cc(M)$-module),
then there exists a bundle of Lie algebras
$\fb(\Prim(A))$ over $M$ such that $\Gc(\fb(\Prim(A)))\cong\Prim(A)$.
Its fiber over $x\in M$ is given by
\[
\fb(\Prim(A))_{x}=\Prim(A)(x)=\Prim(A)/I_{x}\Prim(A)\;,
\]
with the Lie bracket induced from the one on $\Prim(A)$.
The universal enveloping algebra $\Univ(\fb(\Prim(A)))$
is a Hopf $\C(M)$-algebroid and a Hopf algebra over $\C(M)$, while
$\Univ_{c}(\fb(\Prim(A)))=\Cc(M)\ten_{\C(M)}\Univ(\fb(\Prim(A)))$
is a Hopf $\Cc(M)$-algebroid and a Hopf algebra over $\Cc(M)$.
Observe that we have
$\Univ_{c}(\fb(\Prim(A)))=\Univ(\Cc(M),\Prim(A))$.

\begin{prop}\label{prop_injective map from U(P(A))}
Let $A$ be a Hopf $\Cc(M)$-algebroid which is locally free as a $\Cc(M)$-module.
If the $\Cc(M)$-module $\Prim(A)$ is $S$-invariant and locally free of constant finite rank,
then the natural homomorphism $\Univ_{c}(\fb(\Prim(A)))\ra A$
induces an isomorphism of algebras $\Univ_{c}(\fb(\Prim(A)))\ra\D(A)$.
\end{prop}
\begin{proof}
The image of the natural homomorphism
$\nu\!:\Univ_{c}(\fb(\Prim(A)))\ra A$
equals $\D(A)$ by definition, so we only need
to show that $\nu$ is injective. For this, it is sufficient to prove that $\nu$ 
is locally injective. Choose any $x\in M$. Since $\Prim(A)_{x}$
is a free $\C(M)_{x}$-module, 
it follows from PBW-theorem that $\Prim(\Univ(\Cc(M),\Prim(A))_{x})=\Prim(A)_{x}$, which
proves that
$\nu_{x}|_{\Prim(A)_{x}}$ is injective. 
As it is well known, this implies that
$\nu_{x}$ is injective.
\end{proof}

Next, we will use the operators $T_{W,a}$
to define an action of the spectral \'{e}tale groupoid $\Gsp(A)$ on
$\fb(\Prim(A))$. 
Let $a\in G_{w}^{S}(A)$ be good, normalized on an open subset $W$ of $M$.
By Proposition \ref{prop_isomorphism T}, the
map $T_{W,a}$ restricts to an isomorphism
$\Prim(A)|_{V_{W,a}}\ra\Prim(A)|_{W}$ of Lie algebras. Moreover,
by considering $\Prim(A)|_{W}$ as a left $\Cc(V_{W,a})$-module
via $(\tau_{W,a}^{-1})^{*}\!:\Cc(V_{W,a})\ra\Cc(W)$, the map $T_{W,a}$ becomes
an isomorphism of
$\Cc(V_{W,a})$-modules, hence it is given by an
isomorphism of vector bundles
$(\mu_{W,a},\tau_{W,a})\!:\fb(\Prim(A))|_{V_{W,a}}\ra\fb(\Prim(A))|_{W}$.
These maps, for all possible choices of $a$ and $W$,
assemble to a smooth map
\[
\mu\!:\Gsp(A)\times_{M}\fb(\Prim(A))\ra\fb(\Prim(A))\;.
\]
Explicitly, for any $g\in\Gsp(A)(x,y)$ and any
$p\in\fb(\Prim(A))_{x}$ we have
\[
g\cdot p=\mu(g,p)=T_{W,a}(D)(y)\;,
\]
where $a$ is a good element of $G_{w}^{S}(A)$ normalized on $W$
with $y\in W$ such that $a_{y}=g$ and
$D\in\fb(\Prim(A))|_{V_{W,a}}$ satisfies $D(x)=p$.

\begin{prop}
The map $\mu$ defines an action of the groupoid $\Gsp(A)$ on the
bundle of Lie algebras $\fb(\Prim(A))$.
\end{prop}
\begin{proof}
Any $g\in\Gsp(x,y)$ acts as an isomorphism
$\fb(\Prim(A))_{x}\ra\fb(\Prim(A))_{y}$ of Lie algebras by
Proposition \ref{prop_isomorphism T}. 

Represent an arrow $1_{x}\in\Gsp(A)$ by a smooth function
$f\in\Cc(M)$ with $f_{x}=1$. 
Choose a neighbourhood $W$ of $y$ such that $f$ is normalized
on $W$.
For any $p\in\fb(\Prim(A))_{x}$,
represented by $D\in\Prim(A)|_{W}$, we obtain
\[
1_{x}\cdot p=
T_{W,f}(D)(x)=(fDS(f))(x)=(fDf)(x)=D(x)=p\;.
\]

Let $a,b\in G_{w}^{S}(A)$ be good, normalized on open subsets
$W_{a}$ respectively $W_{b}$ of $M$, and let $y\in W_{a}$ and
$z\in W_{b}$ be such that $h=a_{y}$ is an arrow from $x$ to $y$
and $g=b_{z}$ is an arrow from $y$ to $z$. We can assume without loss
of generality
that $V_{W_{b},b}=W_{a}$. The arrow $gh=(ba)_{z}$ is then
represented by the element $ba\in G_{w}^{S}(A)$, which is good
and normalized on $W_{b}$, and for any 
$p\in\fb(\Prim(A))_{x}$,
represented by $D\in\Prim(A)|_{V_{W_{a},a}}$, 
we have
\[
(gh)\cdot p=
(T_{W_{b},ba}(D))(z)=baDS(ba)(z)=b(aDS(a))S(b)(z)=g\cdot(h\cdot p)\;.\qedhere
\] 
\end{proof}

The bundle $\fb(\Prim(A))$ of Lie algebras integrates to a
bundle $\Bsp(A)$ of simply connected Lie groups \cite{DoLaz66}.
Moreover, by the Lie's second theorem for Lie algebroids
\cite{MacXu00,MoMrc03}, we can integrate the action of $\Gsp(A)$
on $\fb(\Prim(A))$ to an action of $\Gsp(A)$ on $\Bsp(A)$. The
corresponding Lie groupoid
\[
\Gsp(A)\ltimes\Bsp(A)
\]
will be
referred to as the 
{\em spectral semidirect product Lie groupoid} associated to $A$.

\subsection{Cartier-Gabriel-Kostant theorem for Hopf algebroids}

The Cartier-Gabriel-Kostant decomposition theorem describes the
structure of Hopf algebras. It states that a Hopf algebra $H$
is isomorphic,
under some conditions,
to the twisted tensor product
$H\cong G(H)\ltimes U(P(H))$, where the group of
grouplike elements $G(H)$ naturally acts on the universal
enveloping algebra $U(P(H))$ of $P(H)$ by conjugation.
For example, this is true if $H$ is cocommutative and
the base field is algebraically closed.
To be able
to similarly decompose a Hopf algebroid over $\Cc(M)$, one first
needs to impose several additional requirements, most of which are
automatically fulfilled in the case of Hopf algebras over $\RR$.

For any Hopf $\Cc(M)$-algebroid $A$ which is locally free as a left
$\Cc(M)$-module and such that $\Prim(A)$ is a $S$-invariant,
locally free left $\Cc(M)$-module of
constant finite rank, we constructed the
spectral semidirect product Lie groupoid $\Gsp(A)\ltimes\Bsp(A)$
and the corresponding Hopf algebroid
\[
\Gsp(A)\ltimes\Univ(\fb(\Prim(A)))\;.
\]
To compare the Hopf
algebroids $\Gsp(A)\ltimes\Univ(\fb(\Prim(A)))$ and $A$, we
define a map
\[
\Theta_{A}\!:\Gsp(A)\ltimes\Univ(\fb(\Prim(A)))\ra A
\]
by
\[
\Theta_{A}\Big(\sum_{i=1}^{n}
D_{i}\com t|_{(a_{i})_{W_{i}}}\Big)=\sum_{i=1}^{n}D_{i}a_{i}\;,
\]
where $a_{i}\in G^{S}_{w}(A)$ is good, normalized on an open subset
$W_{i}$ of $M$, and
$D_{i}\in
\Univ(\fb(\Prim(A)))|_{W_{i}}=
\Univ_{c}(\fb(\Prim(A)))|_{W_{i}}\cong
\D(A)|_{W_{i}}$, for $i=1,\ldots,n$.
The last isomorphism follows from Proposition \ref{prop_injective map
from U(P(A))}.

\begin{prop}
The map $\Theta_{A}\!:\Gsp(A)\ltimes\Univ(\fb(\Prim(A)))\ra A$ is a 
homomorphism of Hopf algebroids.
\end{prop}
\begin{proof}
It is straightforward to check that $\Theta_{A}$ is a well defined
homomorphism of coalgebras over $\Cc(M)$, so it is sufficient to
prove that $\Theta_{A}$ is multiplicative
and commutes with the antipodes.

Let $a,b\in G^{S}_{w}(A)$ be good,
normalized on open subsets $W_{a}$ respectively $W_{b}$ of $M$,
and assume for simplicity that $V_{W_{b},b}=W_{a}$. For any
$D\in\Univ(\fb(\Prim(A)))|_{W_{a}}$ and
$E\in\Univ(\fb(\Prim(A)))|_{W_{b}}$ we have
\[
\Theta_{A}\big((E\com t|_{b_{W_{b}}})(D\com t|_{a_{W_{a}}})\big)
=
\Theta_{A}(E\mu_{b_{W_{b}}}(D)\com t|_{(ba)_{W_{b}}})
=
E\mu_{b_{W_{b}}}(D)ba\;.
\]
On the other hand we have
\[
\Theta_{A}(E\com t|_{b_{W_{b}}})\Theta_{A}(D\com t|_{a_{W_{a}}})=EbDa\;.
\]
Choose $b'\in G^{S}_{w}(A)$ such that $b,b'$ is a good pair.
Since $D\in\Univ(\fb(\Prim(A)))|_{V_{W_{b},b}}$ and
$S(b')b$ equals $1$ on $V_{W_{b},b}$, 
we have
$D=DS(b')b$ and hence
\[
EbDa=EbDS(b')ba
=ET_{b,b'}(D)ba=E\mu_{b_{W_{b}}}(D)ba\;.
\]

To see that $\Theta_{A}$ commutes with antipodes, it is now enough
to observe that this holds true on the subalgebras
$\Univ_{c}(\fb(\Prim(A)))$ (Proposition \ref{prop_injective map from
U(P(A))}) and $\Cc(\Gsp(A))$ (see \cite{Mrc07_1}) which generate
$\Gsp(A)\ltimes\Univ(\fb(\Prim(A)))$ as an algebra.
\end{proof}

Our construction of $\Theta_{A}$ extends the map
$\Cc(\Gsp(A))\ra A$ given in \cite{Mrc07_1} (which is in fact
defined for an arbitrary Hopf algebroid). The latter map is an
isomorphism if and only if $A$ is locally grouplike, that is if
$A_{x}$ is freely generated by $G^{S}(A_{x})$ for every $x\in M$.
In general, we consider $\Cc(\Gsp(A))$ as the best
possible approximation of $A$ by a locally grouplike Hopf
algebroid.

We resume by characterizing the Hopf algebroids $A$ which can be
decomposed as $\Gsp(A)\ltimes\Univ(\fb(\Prim(A)))$. Since we do
not assume $A$ to be cocommutative, $\Theta_{A}$ is not
necessarily an isomorphism (even if $\C(M)\cong\RR$). In his proof
given in \cite{Car07}, Cartier uses cocommutativity to show
that $H$ is freely generated by $G(H)$ as a left $U(P(H))$-module.
For Hopf algebroids over $\Cc(M)$ we need to replace this condition
with a family of local conditions for each $x\in M$.

Let $A$ be a Hopf $\Cc(M)$-algebroid such that $\Prim(A)$ has
trivial anchor. If we localize $\D(A)$ at $x\in M$, we obtain the
$\C(M)_{x}$-algebra $\D(A)_{x}$. Since $\Cc(M)$ and $\D(A)$
commute, the left $\C(M)_{x}$-module
$A_{x}$ naturally becomes a left $\D(A)_{x}$-module, for
every $x\in M$. For any $g\in G^{S}(A_{x})$
we denote by
$A_{x}^{g}$ the (free) left $\D(A)_{x}$-submodule
of $A_{x}$ generated by $g$.

\begin{theo}[Cartier-Gabriel-Kostant]\label{theo_Main theorem}
Let $A$ be a Hopf $\Cc(M)$-algebroid and suppose that $A$ is
locally free as a $\Cc(M)$-module. If
\begin{enumerate}
\item [(i)] the left $\Cc(M)$-module $\Prim(A)$
is $S$-invariant locally free of constant finite rank and

\item [(ii)] $A_{x}=\bigoplus_{g\in G^{S}(A_{x})}A_{x}^{g}$ for
any $x\in M$,
\end{enumerate}
then the map
$\Theta_{A}\!:\Gsp(A)\ltimes\Univ(\fb(\Prim(A)))\ra A$ is an
isomorphism of Hopf $\Cc(M)$-algebroids.
In particular, the Hopf $\Cc(M)$-algebroid $A$
is isomorphic to the Hopf algebroid associated
to its spectral semidirect product Lie groupoid
$\Gsp(A)\ltimes\Bsp(A)$.
\end{theo}

\begin{proof}
The map $\Theta_{A}$ is an isomorphism if and only if
$(\Theta_{A})_{x}$ is an isomorphism for every $x\in M$. 
By definition of the groupoid $\Gsp(A)$ we have
$t^{-1}(x)=G^{S}(A_{x})$ and therefore
\[
\big(\Gsp(A)\ltimes\Univ(\fb(\Prim(A)))\big)_{x}
\cong\!\!\!\bigoplus_{g\in\,G^{S}(A_{x})}\!\!\!
\Univ(\fb(\Prim(A)))_{x}\delta_{g}\;,
\]
where $\delta_{g}\in\Cc(\Gsp(A))$ is a function with germ $1$ at $g$
and support in a small bisection of $\Gsp(A)$. This means that
$\big(\Gsp(A)\ltimes\Univ(\fb(\Prim(A)))\big)_{x}$ is a free left
$\Univ(\fb(\Prim(A)))_{x}$-module with basis
$\{\delta_{g}\,|\,g\in\,G^{S}(A_{x})\}$. If we identify
$\Univ_{c}(\fb(\Prim(A)))$ with $\D(A)$ via $\Theta_{A}$, we can
consider
\[
(\Theta_{A})_{x}\!:(\Gsp(A)\ltimes\Univ(\fb(\Prim(A))))_{x}\ra A_{x}
\]
as a homomorphism of left $\D(A)_{x}$-modules, uniquely determined by
$(\Theta_{A})_{x}(\delta_{g})=g$. It follows
that $(\Theta_{A})_{x}$ is an isomorphism if and only if
$A_{x}$ is a free left $\D(A)_{x}$-module with basis $G^{S}(A_{x})$.
\end{proof}

\noindent
{\bf Acknowledgements:} 
We are grateful to I. Moerdijk, P. Cartier, M. Crainic
and N. Kowalzig for
helpful discussions related to this paper.

\end{document}